\documentclass[a4paper,11pt]{article}
\usepackage{amsmath, amsthm, amssymb}
\usepackage{graphicx}
\usepackage{enumerate}

\newcommand{\R}{{\cal{R}}}
\renewcommand{\P}{{\cal{P}}}

\newcommand{\eps}{{\varepsilon}}

\newcommand{\pmX}{{X^{\pm 1}}}

{\catcode`\|=\active
  \gdef\Set#1{\left\{\:{\mathcode`\|"8000\let|\SetVert #1}\:\right\}}}
{\catcode`\|=\active
  \gdef\Pres#1{\left\langle\:{\mathcode`\|"8000\let|\SetVert #1}\:\right\rangle}}
\def\SetVert{\egroup\;\middle|\;\bgroup}

\newtheorem{theorem}{Theorem}

\newtheorem{remark}[theorem]{Remark}

\newtheorem{observation}[theorem]{Observation}

\theoremstyle{definition}
\newtheorem{definition}[theorem]{Definition}

\title{A note on the structure of $V(6)$ maps}
\author{Uri Weiss \\ uriw@tx.technion.ac.il}

\begin{document}

\maketitle

\begin{abstract}
The purpose of this note is make Theorem 13 in \cite{Wei07} more accessible. Restatements of the theorem already appeared in few of the authors' succeeding works but with no details. We wish in this note to give the necessary details to these restatements.
\end{abstract}

\section{Introduction}

Theorem 13 in \cite{Wei07} is the following theorem:

\begin{theorem}[Cut corners exist in thick equality diagrams] \label{thm:thickDiag}
Assume that $w$ and $v$ are $(a,b)$-equal such that the equality diagram of $aw$ and $vb$ is not thin. Then, either $w$ or $v$ have a cut corner.
\end{theorem}

The theorem essentially deals with the structure of special $V(6)$ diagrams. We aim to give all the necessary details and also to give a full and complete statement of the theorem so it can be used readily. It is assumed that the reader is familiar with the subject combinatorial group theory and has some knowledge of van Kampen diagrams (although, an effort is made to give some basic details and some references).

A note on the presentation: we will usually denote words (over some alphabet) with uppercase letters, $W$, $U$, and $V$. However. in \cite{Wei07} words are denoted with lowercase letters and we keep the original form in citations.

The rest of this note is organized as follows. In Section \ref{sec:preliminaries} we give a brief overview of the subject of van Kampen maps and diagrams and we give the definition of $V(6)$ maps. In Section \ref{sec:restatement} we will restate Theorem \ref{thm:thickDiag} with all the necessary details (including the meaning of ``$(a,b)$-equal'' and ``cut corner''). In Section \ref{sec:newStatment} we give the version of the theorem which only deals with maps. Section \ref{sec:consequences}, the last section, contains few consequences of the theorem.

\section{Preliminaries} \label{sec:preliminaries}

A \emph{map} is a finite planar connected and simply connected 2-complex (see \cite[Chapter V]{LS77}). We name the $0$-cells, $1$-cells, and $2$-cells by \emph{vertices}, \emph{edges}, and \emph{regions}, respectively. Vertices of valence one or two are allowed. Each edge has an orientation, i.e., a specific choice of initial and terminal vertices. Given an edge $e$ we denote by $i(e)$ the initial vertex of $e$ and by $t(e)$ the terminal vertex of $e$. If $e$ is an oriented edge then $e^{-1}$ will denote the same edge but with the reverse orientation. A \emph{path} is a series of (oriented) edges $e_1,e_2,\ldots,e_n$ such that $t(e_j) = i(e_{j+1})$ for $1\leq j < n$. The length of a path $\rho$ (i.e., the number of edges along $\rho$) is denoted by $|\rho|$. Paths of length zero are allowed; these paths consist of a single vertex. If $\rho$ is the path $e_1 \cdots e_n$ then we denote by $\rho^{-1}$ the path $e_n^{-1} \cdots e_1^{-1}$. If $\rho$ is a path that decomposes as $\rho=\rho_1\rho_2$ then $\rho_1$ is a \emph{prefix} of $\rho$ and $\rho_2$ is a \emph{suffix} of $\rho$. The term \emph{neighbors}, when referred to two regions of a map, means that the intersection of the regions' boundaries contains an edge; specifically, if the intersection contains only vertices, or is empty, then the two regions are not neighbors. \emph{Boundary edges} are edges in the boundary of the map. \emph{Boundary regions} are regions with outer boundary, i.e., the intersection of their boundary and the map's boundary contains at least one edge. The outer boundary of of a boundary region $D$ in a map $M$ is simply $\partial D \cap \partial M$; the inner boundary of $D$ is $\partial D \setminus \left(\partial D \cap \partial M\right)$. \emph{Inner regions} are regions which are not boundary regions. \emph{Proper boundary regions} are boundary regions which have the property that removing their interior and all their boundary edges keeps the map connected. If $D$ is a proper boundary region of $M$ then $\partial D \cap \partial M$ contains only one connected component which contain edges. A \emph{boundary path} is a path in the boundary of the map.

Let $M$ be a map with boundary cycle $\mu\sigma^{-1}$. We say that $M$ is a \emph{$(\mu,\sigma)$-thin} if every region $D$ of $M$ has at most two neighbors and both $\partial D\cap \mu$ and $\partial D\cap \sigma$ are non-empty. See Figure \ref{fig:thinMap} for an illustration of such a map.

\begin{figure}[ht]
\centering
\includegraphics[totalheight=0.18\textheight]{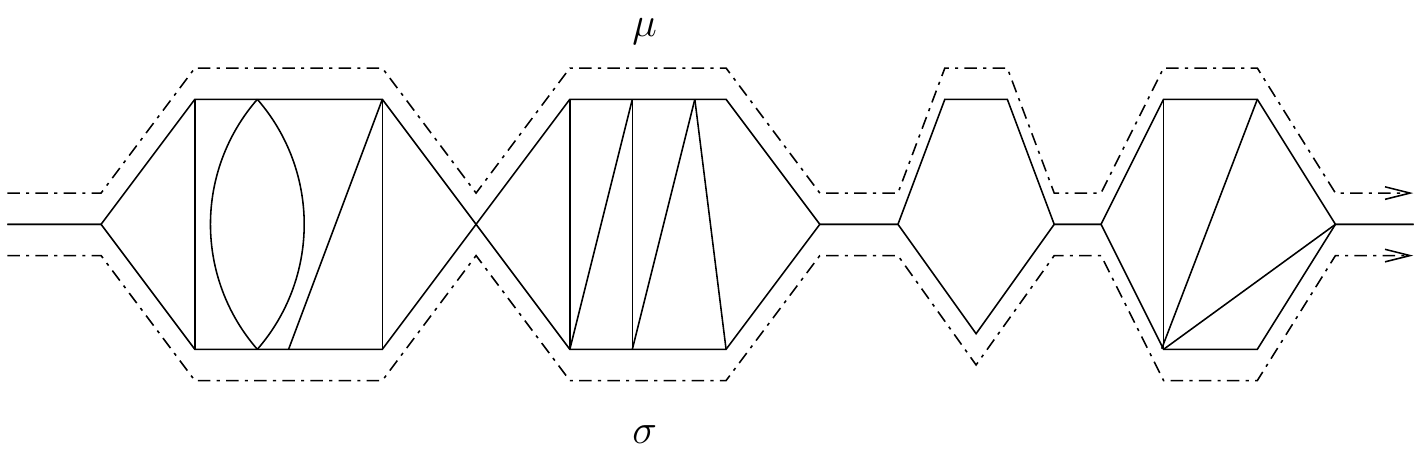}
\caption{Thin map} \label{fig:thinMap}
\end{figure}

\begin{definition}[$V(6)$ diagram] \label{def:v6Diag}
A map $M$ is a called a \emph{$V(6)$ map} if the following holds: suppose $D$ is an \emph{inner region} then:
\begin{enumerate}
 \item $D$ has at least four neighbors.
 \item If the boundary of $D$ contains a vertex of
 valence three then $D$ has at least six neighbors.
\end{enumerate}
\end{definition}

Given a finite presentation $\P=\Pres{X|\R}$, an \emph{$\R$-diagram} (or simply \emph{diagram} if the presentation is known from the context) is a map---called the \emph{underling map} of the diagram---where its edges are labelled by elements of $(\pmX)^*$ and the boundary of every region is labelled by elements of the symmetric closure of $\R$. Suppose that a group $G$ is presentated by $\P=\Pres{X|\R}$. van Kampen theorem \cite[Chapter V]{LS77} states that a word $W$ over $\pmX$ presents the identity of $G$, if and only if there is an $\R$-diagram with a boundary cycle labelled by $W$. A diagram is a $V(6)$ diagram if its underlying map is a $V(6)$ map. 

\section{Summary of original statement} \label{sec:restatement}

Theorem 13 in \cite{Wei07} appears in Section 3 which spans over pages 801-803. The section assumes that there is a fixed $V'(6)$ presentation $\P=\Pres{X|\R}$, where $\R$ is symmetrically closed and all its elements are freely reduced. The definition of $V'(6)$ presentation is given in page 798: 

\begin{definition}[$V'(6)$ presentations]
Let $\P=\Pres{X|\R}$ be a finite presentation of group $G$, where $\R$ is symmetrically closed and all its elements are freely reduced. A piece is a non-trivial word $U$, such that there are two different relators $R_1$ and $R_2$ in $\R$ with $U$ as their prefix. We say that $\P$ is a $V(6)$ presentation if for every relator $R\in\R$, one of the
following holds:
\begin{enumerate}
    \item Every decomposition of $R$ into pieces contains at least four pieces and if $R',R''\in\R$ then one of the three words $RR'$,$R'R''$ or $R''R$ is freely reduced.
    \item Every decomposition of $R$ into pieces contains at least six pieces.
\end{enumerate}
Presentation which are $V(6)$ presentation and every piece is of length one will be denoted by $V'(6)$. 
\end{definition}

There are several definition that are used in the statement of the theorem; we give them next for completeness. Two words, $W$ and $V$, are said to be ``$(a,b)$-equal'' (page 798) if $aW$ and $Vb$ present the same element in the group where $a$ and $b$ are elements of $\pmX\cup\Set{\eps}$ ($\eps$ is the empty word). Equality diagram for two words, $W$ and $V$, is a diagram $M$ which has a boundary cycle that is labelled by $WV^{-1}$ (page 799). A thin equality diagram for two words, $W$ and $V$, is an equality diagram which is $(\mu,\sigma)$-thin where $\mu$ is labelled by $W$ and $\sigma$ is labelled by $V$ (page 799). 

Next, we clarify the meaning of ``$w$ or $v$ have a cut corner'' in the statement of the theorem. In \cite{Wei07}, if a path $\mu$ is labelled by a word $W$ then letter $W$ may present both the path and the word (page 799). Thus, in the phrase above we have two paths, $\mu$ and $\sigma$, labelled by $w$ and $v$, respectively, which contain a cut corner. The definition of a path containing a cut corner follows (based on the definition which appear in page 802). For sake of brevity and simplicity, we give a version of the definition which differ from the original definition. See the note after the definition for the list of changes.

\begin{definition}[Cut Corners] \label{def:cutConers}
Let $M$ be a diagram and let $D$ be a \emph{proper} boundary region. Let $\mu=e_1 e_2 \cdots e_n$ be a boundary path of $M$, where this path contains all of the outer boundary of $D$ which we denote by $\mu_D$. Assume that $\mu_D$ contains the edges $e_\ell e_{\ell+1} \cdots e_{\ell+r}$, $1 \leq \ell \leq \ell+r \leq n$, and assume that the inner boundary of $D$, $\delta_D$, has $s$ edges. The edges $e_\ell$ and $e_{\ell-1}$ intersect in a vertex which we denote by $i(\mu_D)$. We say that \emph{$D$ is a cut corner for $\mu$}, if one of the following conditions hold:
\begin{enumerate}
    \item[T1.] $s<r$.
    \item[T2.] $s=r=2$, $\ell>1$ and $i(\mu_D)$ is of valence three.
    \item[T3.] $s=r=3$, $\ell>1$, $i(\mu_D)$ is of valence three and $e_{\ell-1}$ is on the boundary of an adjacent boundary region $E$ with at most five edges.
    \item[T4.] $s=r=3$, $\ell>2$, $i(\mu_D)$ is of valence three and $e_{\ell-2} e_{\ell-1}$ is on the boundary of an adjacent boundary region $E$.
\end{enumerate}
The phrases ``$\mu$ contains a cut corner $D$'' and ``$D$ is a cut corner for $\mu$'' will be used interchangeably.
\end{definition}

\begin{figure}[ht]
\centering
\includegraphics[totalheight=0.25\textheight]{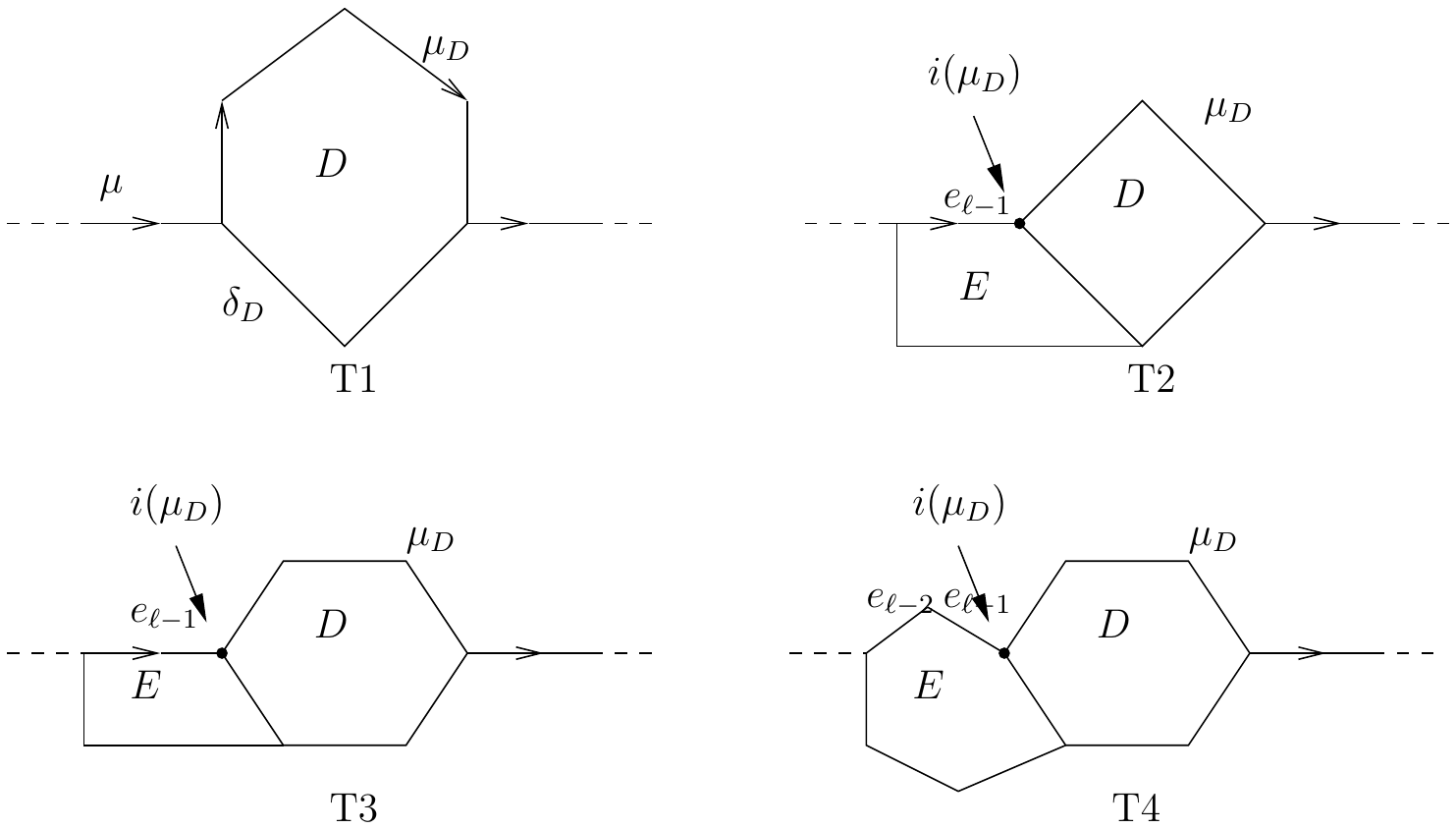}
\caption{Cut Corner} \label{fig:CC}
\end{figure}

\begin{remark}
As written above, the definition of cut corners is different from the original definition in \cite{Wei07}. We list the changes, which are only syntactic in nature, as can be easily verified.
\begin{enumerate}
 \item We avoided using the same name for paths and labells. This is done by removing the unnecessary reference to the labels. The reader can check that the actual labels of the paths are not important to the definition. 
 \item Some of the names of the object that appear in the definition where changed to reflect the terminology of this note.
 \item The region $D$ is assumed to be a \emph{proper} boundary region. In the original version this was not explicitly assumed although it follows from the assumptions of the definition.
 \item We removed the equality ``$s=n-r$'' which appeared in the original statement (this is a typo in the original version).
\end{enumerate}
\end{remark}

We end this section with a restatement of the main theorem and some observations. In the version below we tried to use as few as possible definitions.

\begin{theorem} \label{thm:diagVer}
Let $\P=\Pres{X|\R}$ be a $V'(6)$ presentation of a group $G$, let $W$ and $V$ be words over $\pmX$, and let $a$ and $b$ be elements of $\pmX\cup\Set{\eps}$. Assume that $aW$ and $Vb$ present the same element in $G$ and let $M$ by a diagram with boundary cycle $\xi\mu\tau^{-1}\sigma^{-1}$ such that $\xi$ is labelled by $a$, $\mu$ is labelled by $W$, $\tau$ is labelled by $b$, and $\sigma$ is labelled by $V$. Assume that the diagram $M$ is \emph{not} $(\xi\mu,\sigma\tau)$-thin. Then, either the path $\mu$ or the path $\sigma$ contains a cut corner.
\end{theorem}

\begin{observation}
We list few observations regarding the theorem which will be important later:
\begin{enumerate}
 \item A diagram over a $V'(6)$ presentation is a $V(6)$ diagram (page 802).
 \item Since all pieces of the presentation are of length one there are no inner vertices of valence two. 
 \item All edges of the diagram are labelled by a generator (this is assumed in the definition of a diagram in page 799). Consequently, the lengths of $\xi$ and $\tau$ are at most one.
\end{enumerate}
\end{observation}

\section{A new statement} \label{sec:newStatment}

In this section we shift the focus to maps (instead of diagrams). The proof of Theorem 13 in \cite{Wei07} is given in Section 4 (pages 804-819). As can be verified, the proof does not concern with the actual labels of the diagram. Thus, the proof deals only with the structure of special $V(6)$ diagrams under certain conditions. Our goal in this section is to make the assumption on the structure of the diagram explicit and finally restate the theorem as a theorem on maps.

Let $M$ be a diagram over a $V'(6)$ presentation for which the conditions of the theorem hold. As we observed, the map is a $V(6)$ map and has no inner vertices of valence two. The boundary regions have also a special structure. Since each edge is labelled by a generator and the presentation is a $V'(6)$ presentation it follows that one of following two conditions hold for a boundary region $D$ of $M$: (1) $\partial D$ contains at least four edges; (2) if $\partial D$ contains an inner vertex of valence three then $\partial D$ contains at least six edges. A $V(6)$ map with no inner vertices of valence two for which the above two conditions hold for every boundary regions is called \emph{proper $V(6)$ map}. It is clear that the underling map of the diagram $M$ in Theorem \ref{thm:diagVer} is a proper $V(6)$ map. One can verify by going over the proof (which is too long to reproduce here) that the only needed assumption for the theorem to hold is the assumption that the map is a proper $V(6)$ map. Thus, it follows that following theorem hold: 

\begin{theorem} \label{thm:mapVer}
Let $M$ be a proper $V(6)$ map with boundary $\xi\mu\tau^{-1}\sigma^{-1}$ such that:
\begin{enumerate}[(a) ]
 \item $|\xi|\leq1$ and $|\tau|\leq1$.
 \item $M$ is not $(\xi\mu,\sigma\tau)$-thin.
\end{enumerate}
Then, either the path $\mu$ or the path $\sigma$ contains a cut corner.
\end{theorem}

\section{Consequences} \label{sec:consequences}

In this section we state two private cases of Theorem \ref{thm:mapVer}. The full statement appeared in \cite{Wei10a, Wei10b}. In \cite{Wei10c} the following theorem appears:

\begin{theorem}
Let $M$ be a $C(7)$ map with boundary $\xi\mu\tau^{-1}\sigma^{-1}$ such that $|\xi|\leq1$ and $|\tau|\leq1$ and $M$ is not $(\xi\mu,\sigma\tau)$-thin. Then, there is a proper boundary region $D$ which its outer boundary is contained in the path $\mu$ or the path $\sigma$ and has at most three neighbors.
\end{theorem}
\begin{proof}[Proof (sketch).]
By removing inner vertices of valence two and adding vertices of valence two in $\mu$ or $\sigma$ we can assume that $M$ is a proper $C(7)$ diagram. Namely, every region $D$ has the property that $\partial D$ contains at least seven edges and there are no inner vertices of valence two. We make sure to add the minimal number of vertices to $\mu$ and $\sigma$ for this property to hold (or, possibly, remove vertices). $M$ is also a proper $V(6)$ map. Consequently, $\mu$ or $\sigma$ contains a cut corner $D$. Since this is a $C(7)$ map the region $D$ is a cut corner of type T1 and so it has more outer edges then inner edges. Let $s$ be the number of inner edges and $r$ the number of outer edges of $D$. So, $s < r$. We need to show that $s \leq 3$. If that is not the case then $s + r \geq 4 + 5 = 9$. This is not impossible since we can remove some vertices from the outer boundary of $D$ while keeping the map a proper $C(7)$ map. Hence, it follows that $s \leq 3$ as needed and $D$ has at most three neighbors.
\end{proof}

Next, another private case of Theorem \ref{thm:mapVer} which appeared in \cite{Wei10d}. This private case deals with proper $C(4)\&T(4)$ maps. A proper $C(4)\&T(4)$ map is a map with no inner vertices of valence less than four and for which each regions has at least four edges in its boundary. These maps are special type of proper $V(6)$ maps. We start with the definition of ``thick configuration'':

\begin{definition}[Thick configurations] \label{def:thickConf}
Let $M$ be a proper $C(4) \& T(4)$ map and let $\alpha$ be a path on the boundary of $M$. A \emph{thick configuration} in $\alpha$ is a sub-diagram $N$ of $M$ where one of the following holds:
\begin{enumerate}
	\item $N$ contains single region $D$ with $\partial D = \mu \sigma^{-1}$ such that $\mu = \partial D \cap \alpha$ and $|\mu| > |\sigma|$. See Figure \ref{fig:thickConf}(a).
	\item $N$ has connected interior and consists of two neighboring regions $D_1$ and $D_2$. The boundary of $D_2$ decomposes as $\partial D_2 = \mu\sigma^{-1}$ where $|\mu|=|\sigma|=2$, $\mu$ is a sub-path of $\alpha$, and $\sigma$ contains only inner edges. The boundary of $D_1$ contains an outer edge $e$ such that $e \mu$ is a sub-path of $\alpha$. See Figure \ref{fig:thickConf}(b).
\end{enumerate}
If there is a thick configuration along $\alpha$ then we say that \emph{$\alpha$ contains a thick configuration}.

\begin{figure}[ht]
\centering
\includegraphics[totalheight=0.20\textheight]{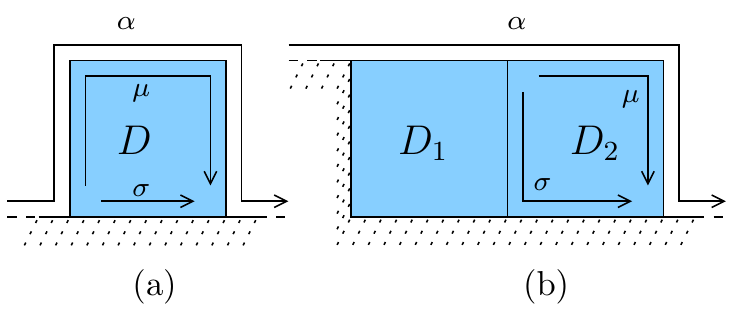}
\caption{Thick Configurations}\label{fig:thickConf}
\end{figure}

\end{definition}

The following theorem characterizes when a proper $C(4) \& T(4)$ diagram is thin (and as we said, it is a special case of Theorem \ref{thm:mapVer}).

\begin{theorem} \label{thm:noThickToThinDiag}
Let $M$ be a proper $C(4) \& T(4)$ map with boundary cycle $\sigma \alpha \tau^{-1} \beta^{-1}$ such that $|\sigma|\leq 1$ and $|\tau|\leq1$. If $\alpha$ and $\beta$ do not contain thick configurations then $M$ is $(\sigma\alpha,\beta\tau)$-thin.
\end{theorem}
\begin{proof}
Since the map is a proper $C(4)\&T(4)$ map the only possible cut corners are of type T1 and T2 which are the thick configurations defined above.
\end{proof}


\end{document}